\documentclass[11pt,twoside,reqno]{amsart}

\usepackage{orcidlink}
\usepackage{mathtools}
\usepackage{todonotes}
\usepackage{microtype}
\usepackage[noadjust]{cite}
\usepackage[OT1]{fontenc}
\usepackage{type1cm}
\usepackage{amssymb}
\usepackage{dsfont}
\usepackage{enumitem}
\usepackage{comment}
\usepackage{xcolor}
\usepackage{subcaption}
\usepackage{lmodern}

\usepackage{geometry}
\usepackage{hyperref}
\hypersetup{
  colorlinks=true,
  linkcolor=black,
  anchorcolor=black,
  citecolor=black,
  filecolor=black,      
  menucolor=red,
  runcolor=black,
  urlcolor=black,
}
\usepackage[capitalize,noabbrev,nameinlink]{cleveref}

\makeatletter
\AddToHook{cmd/appendix/before}{\def\cref@section@alias{appendix}}
\makeatother

\theoremstyle{plain}
\newtheorem{theorem}{Theorem}[section]

\newtheorem{proposition}[theorem]{Proposition}

\newtheorem{lemma}[theorem]{Lemma}

\newtheorem{ques}{Question}

\theoremstyle{remark}
\newtheorem{remark}[theorem]{Remark}

\theoremstyle{definition}

\newtheorem*{definition*}{Definition}

\newtheorem*{question*}{Question}

\AddToHook{env/proposition/begin}{\crefalias{theorem}{proposition}}
\AddToHook{env/lemma/begin}{\crefalias{theorem}{lemma}}
\AddToHook{env/corollary/begin}{\crefalias{theorem}{corollary}}
\AddToHook{env/claim/begin}{\crefalias{theorem}{claim}}
\AddToHook{env/conjecture/begin}{\crefalias{theorem}{conjecture}}
\AddToHook{env/remark/begin}{\crefalias{theorem}{remakr}}
\AddToHook{env/example/begin}{\crefalias{theorem}{example}}
\AddToHook{env/remark/begin}{\crefalias{theorem}{remark}}

\crefformat{equation}{(#2#1#3)}
\crefformat{figure}{#2Figure #1#3}
\crefformat{enumi}{(#2#1#3)}
\crefformat{theorem}{#2Theorem~#1#3}
\crefformat{lemma}{#2Lemma~#1#3}
\crefformat{proposition}{#2Proposition~#1#3}
\crefformat{corollary}{#2Corollary~#1#3}
\crefformat{claim}{#2Claim~#1#3}
\crefformat{section}{#2Section~#1#3}
\crefformat{subsection}{#2Section~#1#3}
\crefformat{remark}{#2Remark~#1#3}

\numberwithin{equation}{section}

\newcommand{\R}{\mathbb{R}}

\newcommand{\Z}{\mathbb{Z}}
\newcommand{\N}{\mathbb{N}}

\newcommand{\eps}{\varepsilon}

\newcommand{\mtt}[1]{\mathtt{#1}}
\newcommand{\mb}[1]{\mathbf{#1}}
\newcommand{\omtt}[1]{\overline{\mathtt{#1}}}
\newcommand{\Bad}{\mathbf{Bad}}

\renewcommand{\geq}{\geqslant}
\renewcommand{\leq}{\leqslant}

\DeclareMathOperator{\udimloc}{\overline{dim}_{loc}}
\DeclareMathOperator{\ldimloc}{\underline{dim}_{loc}}

\DeclareMathOperator{\dimh}{dim_H}

\DeclareMathOperator{\diml}{dim_L}
\DeclareMathOperator{\dimml}{dim_{ML}}

\DeclareMathOperator{\Tan}{Tan}

\DeclareMathOperator{\diam}{diam}

\usepackage{tabularx}
\usepackage{booktabs}

\begin{document}

\title[Badly approximable points on non-linear carpets]{Badly approximable points \\ on non-linear carpets}

\author{Roope Anttila}
\address[Roope Anttila]
        {University of St Andrews \\ 
         Mathematical Institute\\ 
         St Andrews KY16 9SS\\ 
         Scotland}
\email{ra216@st-andrews.ac.uk}

\author{Jonathan M. Fraser}
\address[Jonathan M. Fraser]
        {University of St Andrews \\ 
         Mathematical Institute\\ 
         St Andrews KY16 9SS\\ 
         Scotland}
\email{jmf32@st-andrews.ac.uk}

\author{Henna Koivusalo}
\address[Henna Koivusalo]
        {School of Mathematics \\ 
         Fry Building \\ 
         Woodland Road \\ 
         Bristol BS8 1UG \\
         United Kingdom}
\email{henna.koivusalo@bristol.ac.uk}

\subjclass[2020]{} 
\date{\today}
\thanks{RA and JMF were supported by the EPSRC, grant no. EP/Z533440/1.  JMF was also supported by a Leverhulme Trust Research Project Grant (RPG-2023-281). We thank Lawrence Lee for drawing our attention to Question 5.2 from \cite{Das2019}.}

\begin{abstract}
 The badly approximable points in $\R^d$ are those for which Dirichlet's approximation theorem cannot be improved by more than a constant, that is, they are the points most difficult to approximate by rational vectors. An important problem in Diophantine approximation is to determine when the set of badly approximable points intersects a given set in full dimension. We find the first class of non-linear non-conformal attractors for which this full intersection property holds, thus answering a question of Das--Fishman--Simmons--Urba\'nski from 2019.  We also provide a formula for the Hausdorff dimension of these attractors which is of independent interest.
\end{abstract}

\maketitle

\tableofcontents

\section{Introduction}\label{sec:intro}
Diophantine approximation is the quantitative study of how well irrational points can be approximated by rational points. This fundamental problem has many variants which impinge on many different areas of mathematics, including number theory, ergodic theory, and fractal geometry. The most classical result in the field is Dirichlet's approximation theorem, which gives a sharp approximation rate for all points in $\R^d$. The points for which Dirichlet's theorem cannot be improved by more than a constant are called \emph{badly approximable}: they are the  points $\mb{x}\in\R^d$, for which there is a constant $c>0$ such that for all $\mb{p}\in\Z^d$ and $q \in \mathbb{N}$,
\begin{equation*}
    \left\|\mb{x}-\frac{\mb{p}}{q}\right\|\geq \frac{c}{q^{1+1/d}}.
\end{equation*}

An important question in Diophantine approximation is to understand the size and distribution of badly approximable points. A classical theorem of Khinchine shows that the set of badly approximable points in $\R^d$, which we denote by $\Bad_d$, has Lebesgue measure zero, but it was shown in \cite{Schmidt1969}, using a game theoretic approach now known as   Schmidt's game, that $\Bad_d$ is of full Hausdorff dimension $d$. Moreover, the badly approximable points should be very uniformly spread out, and in particular it is expected that badly approximable points on a fractal set $X$, which is not designed to specifically avoid them, should have full Hausdorff dimension, that is, one expects
  \[
\dimh X \cap \Bad_d = \dimh X.
  \]
This has been verified for many important classes of fractal sets $X$, such as for certain Ahlfors regular sets \cite{Kleinbock2005,Fishman2009} (building on ideas from \cite{Kleinbock2005a}) and for some self-affine carpets satisfying necessary non-concentration conditions \cite{Das2019}. A useful approach is to use a variant of Schmidt's game and the lower dimension. There are many equivalent ways to define the lower dimension, but the one that is convenient for us is the version defined using weak tangents. Recall that a compact set $T$ is called a \emph{weak tangent} of a compact set $X$ (denoted by $T\in\Tan(X)$), if there are $x_n\in X$, and $0<r_n\leq\diam(X)$, such that
  \begin{equation*}
        \frac{X\cap B(x_n,r_n)-x_n}{r_n}\to T,
  \end{equation*}
in the Hausdorff distance. We define the \emph{lower dimension} of a non-empty compact set $X \subseteq \R^d$ by
\begin{equation*}
    \diml X = \min\{\dimh T\colon T\in\Tan(X)\}.
\end{equation*}
The minimum was shown to exist in \cite[Theorem 1.1]{Fraser2019a}, where the equivalence of this definition with the more common definition for lower dimension was proved as well. We refer the reader to \cite{Fraser2019a,Fraser21} for more background on the lower dimension. Note that $\diml X \leq \dimh X$, since, by our definition, (a translated copy of) $X$ is a weak tangent of itself. One can further require that $r_n\searrow 0$ for the minimising weak tangent, but we do not need to make this assumption.  The lower dimension identifies the thinnest parts of the set $X$ and is dual to the Assouad dimension, which identifies the thickest parts of the set. 

 The Schmidt game approach yields the following estimate, which we will use in this paper. This result was proved in \cite[Theorem 3.1]{Fishman2009}, see also \cite[Corollary 2.6]{Das2019} and \cite[Section 14.2]{Fraser21}.
  \begin{proposition} \label{schmidt}
      Suppose $X \subseteq \R^d$ is closed and \emph{hyperplane diffuse} in the sense that there exists $\beta>0$ such that for all $R \in (0,1), x \in X$ and affine hyperplanes $V \subseteq \R^d$,
      \[
B(x,R) \cap X \setminus V_{\beta R} \neq \emptyset
      \]
      where $V_\eps$ denotes the open $\eps$-neighbourhood of $V$.  Then
      \[
     \dimh X \cap \Bad_d \geq  \diml  X.
      \]
  \end{proposition}

  The hyperplane diffusivity assumption essentially means that $X$ cannot be too concentrated on hyperplanes.  This condition is  necessary, as the following example shows.  Let 
  \[
  X=\{0\} \times [0,1]^{d-1} \subseteq \R^{d}.
  \]
  Then, Dirichlet's theorem (in $\R^{d-1}$) gives that for all $\mb{x}\in X$ for infinitely many  $\mb{p}\in\Z^d$ and $q \in \mathbb{N}$, 
\[
    \left\|\mb{x}-\frac{\mb{p}}{q}\right\|\leq \frac{c}{q^{1+1/(d-1)}}= o\left(\frac{1}{q^{1+1/d}}\right)
\]
and therefore $X \cap \Bad_d = \emptyset$ but $\diml X = d-1$.
\begin{remark}
The hyperplane diffusivity assumption can be removed when $d=1$ since a set $X \subseteq \R$ which fails to be hyperplane diffuse, necessarily satisfies $\diml X = 0$.  This can be proved straight from the definitions. More generally, if $X \subseteq \R^d$ fails to be hyperplane diffuse, then 
\begin{equation*}
    \diml X \leq  d-1.
\end{equation*}
Indeed, if $X$ fails to be hyperplane diffuse, then for all $\beta=1/n>0$ with $n \in \mathbb{N}$ there exists $R_n \in (0,1)$, $x_n \in X$, and an affine hyperplane $V^n$ such that
\[
B(x_n,R_n) \cap X \subseteq V^n_{ R_n/n}.
\]
But then
\[
\frac{B(x_n,R_n) \cap X -x_n}{R_n} \subseteq \frac{V^n_{R_n/n} -x_n}{R_n}.
\]
Noting that the right hand side is contained in $W^n_{1/n}$ for some affine hyperplane $W^n$ parallel to $V^n$, with distance at most $\frac{1}{n}$ from the origin, by taking a convergent subsequence of the sequence of sets on the left, we obtain $T \in \textup{Tan}(X)$ for which $T \subseteq W$ for some linear hyperplane $W$. Then $\diml X \leq \dimh T \leq \dimh W = d-1$.
\end{remark}

  With \cref{schmidt} in mind, one strategy to estimate $\dimh X \cap \Bad_d$ is to estimate $\diml X$.  Unfortunately, in many cases of interest, $\diml X < \dimh X$ and so a more nuanced strategy is to search for closed and hyperplane diffuse subsets of $X$ with large lower dimension.  This motivates the \emph{modified lower dimension}, defined by
  \[
\dimml X = \sup\{ \diml Y : Y \subseteq X\}.
  \]
  Then, for all closed sets, 
\[
\diml X \leq \dimml X \leq \dimh X.
\]
Bedford--McMullen carpets (see \cite{bmsurvey})   are an important class of self-affine set and are in some sense the most basic case of the family of sets we consider in this work.  They are defined via an iterated functions system (IFS)---which we define in \cref{sec:carpets}---but can be thought of as being defined by selecting a subset of rectangles from  a uniform but anisotropic grid imposed on the unit square and then iterating the chosen pattern \emph{ad infinitum}.  One may therefore talk about maps (in the IFS) corresponding to chosen rectangles in the grid.  This is a useful picture to keep in mind later. Bedford--McMullen carpets are non-conformal fractals  in the sense that they are invariant under a non-conformal IFS (or a non-conformal expanding map) and this comes from the anisotropy in the defining grid: upon iteration, the rectangles become increasingly eccentric. On the other hand, they are linear fractals since the defining IFS is affine, or equivalently, the associated expanding map is piecewise linear. It is known (see \cite{Das2019, Fraser21}) that Bedford--McMullen carpets $X$ satisfy
\[
\dimml X = \dimh X
\]
and, moreover, provided there are at least two columns used and at least one column with two maps, we have
 \[
     \dimh X \cap \Bad_2 =  \dimh  X
      \]
      see \cite{Das2019}.  As such, the desired number theoretic conclusion is reached for this class of non-conformal fractals.   The latter assumption is related to the hyperplane diffusivity constraint. In the proof of this result, the linearity of the underlying dynamics was used in a fundamental way and this motivated Das--Fishman--Simmons--Urba\'nski \cite{Das2019} to pose the following problem:
\begin{ques}[\cite{Das2019} Question 5.2]\label{ques:dasetal}
    Can Schmidt’s game be used to show that $\Bad_d$ has full dimension in some fractal defined by a dynamical system which is both non-conformal and non-linear?
\end{ques}

We answer this question in the affirmative in this paper, see \cref{thm:bad}.  Our main results pertain to a family of non-linear non-conformal carpets, and we obtain their Hausdorff dimensions as well as the Hausdorff dimension of their intersection with $\Bad_2$, see \cref{thm:variational-principle}, \cref{thm:hausdorff-equals-ml}, and \cref{thm:bad}. We also briefly discuss another class of dynamically defined fractals, parabolic Cantor sets, for which the Schmidt game--lower dimension approach bears fruit, see \cref{parabolic}.

\section{Non-linear carpets}\label{sec:carpets}
\begin{figure*}
    \centering
    \begin{subfigure}{0.45\linewidth}
        \centering
        \includegraphics[width=0.9\linewidth]{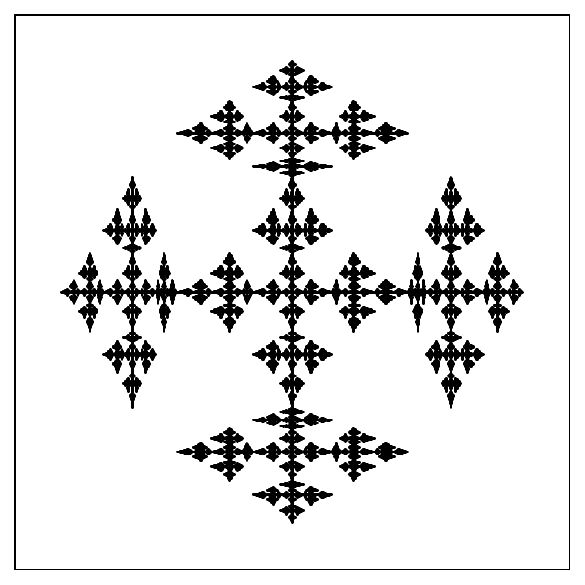}
    \end{subfigure}
    \begin{subfigure}{0.45\linewidth}
        \centering
        \includegraphics[width=0.9\linewidth]{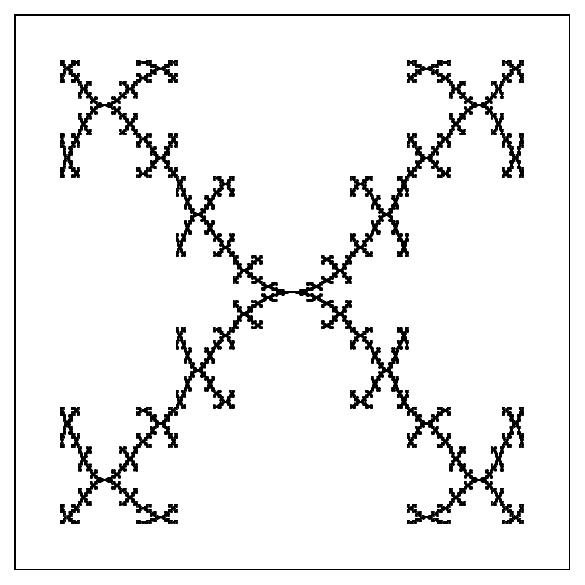}
    \end{subfigure}    
    \caption{Two non-linear carpets which fit in our framework}
    \label{fig:non-linear}
\end{figure*}
Before stating our results precisely, let us define our setup. Recall that a function $f\colon A\to A$, where $A\subset \R^d$, is a \emph{contraction} if $\|f(x)-f(y)\|<\|x-y\|$, for all $x,y\in A$, and a \emph{uniform contraction}, if there exists $0<\lambda<1$, such that $\|f(x)-f(y)\|\leq \lambda\|x-y\|$, for all $x,y\in A$. A collection of contractions $(f_i)_{i\in\Lambda}$, where $\Lambda$ is a finite index set and for all $i\in\Lambda$, $f_i\colon A\to A$, for some compact set $A\subset \R^d$, is called an \emph{iterated function system (IFS)}. The \emph{attractor} of an IFS is the unique non-empty compact set $X\subset A$, which satisfies
\begin{equation*}
    X=\bigcup_{i\in\Lambda}f_i(X).
\end{equation*}
One often imposes various separation conditions on IFSs to make their study easier. The one relevant to the current work is the \emph{open set condition (OSC)}, which the IFS is said to satisfy if there exists an open set $U\subset A$, such that $f_i(U)\subset U$ for all $i\in\Lambda$ and $f_i(U)\cap f_j(U)=\emptyset$, for all $i\ne j$.

In this paper we are mainly interested in a class of non-linear IFSs defined in the following way. Let $\Lambda_1$ and $\Lambda_2$ be finite index sets and let $(f_i)_{i\in\Lambda_1}$ and $(g_j)_{j\in\Lambda_2}$ each be a \emph{self-conformal IFS} on $[0,1]$, that is, each $f_i$ is a $C^{1+\alpha}$ uniform contraction on $[0,1]$ and each $g_j$ is a $C^{1+\beta}$ uniform contraction on $[0,1]$. We call these IFSs the \emph{coordinate IFSs}. Let $\Lambda\subset \Lambda_1\times \Lambda_2$ and consider the planar IFS $(S_{i,j}\coloneqq(f_i,g_j))_{(i,j)\in\Lambda}$. These IFSs can be thought of as non-linear analogues of diagonal self-affine IFSs. Due to this connection, and to emphasise the fact that we allow for non-linear maps in the IFS, we call the attractor of an IFS in the class above a \emph{non-linear carpet}. We assume going forward that a non-linear carpet $X$ satisfies $\diam(X)=1$; this can always be achieved with a rescaling, which does not affect our results. We emphasise that $\Lambda$ may be a strict subset of $\Lambda_1\times \Lambda_2$, so, in general, the inclusion $X\subset X_1\times X_2$, where $X_1$ and $X_2$ are the attractors of $(f_i)_{i\in\Lambda_1}$ and $(g_j)_{j\in\Lambda_2}$, respectively, is strict.

In the generality described above, the class of non-linear carpets is very large; for example, it includes all attractors of diagonal self-affine IFSs, without the need to impose any grid structure. For our methods to work, we need the IFS to have a grid-like structure. The version we find suitable is the \emph{coordinate OSC}, which the non-linear carpet is said to satisfy if both of the coordinate IFSs $(f_i)_{i\in\Lambda_1}$ and $(g_j)_{j\in\Lambda_2}$ satisfy the OSC. Note that this assumption is strictly stronger than assuming that the planar IFS $(S_{i,j})_{(i,j)\in\Lambda}$ satisfies the OSC.

\subsection{Main results}
To state our results, let us briefly recall the symbolic space underlying the IFS; for more precise definitions, see \cref{sec:symbolic_spaces}. We let $\Sigma=\Lambda^{\N}$ denote the natural symbolic space associated to the non-linear carpet $X$, and let $\pi\colon \Sigma\to X$ denote the \emph{natural projection} which is defined for $(\mtt{i},\mtt{j})\coloneqq (i_1i_2\cdots,j_1,j_2\cdots)\in\Sigma$ by
\begin{equation*}
    \pi(\mtt{i},\mtt{j})=\lim_{n\to\infty}S_{i_1,j_1}\circ S_{i_2,j_2}\circ\ldots\circ S_{i_n,j_n}(0).
\end{equation*}
Natural dynamics on $\Sigma$ are given by the left shift $\sigma\colon\Sigma\to\Sigma$, defined by
\begin{equation*}
    \sigma(\mtt{i},\mtt{j})=(i_2i_3\cdots,j_2j_3\cdots),
\end{equation*}
for all $(\mtt{i},\mtt{j})=(i_1i_2\cdots,j_1j_2\cdots)\in\Sigma$. We abuse terminology slightly by saying that a Borel probability measure $\mu$ on $X$ is ergodic if there is a Borel measure $\nu$ on $\Sigma$, which is ergodic with respect to $\sigma$, such that $\mu=\pi_*\nu$.

Our first result is a variational principle for the Hausdorff dimension of the non-linear carpets defined above, which is of independent interest.
\begin{theorem}\label{thm:variational-principle}
    If $X$ is a non-linear carpet which satisfies the coordinate OSC, then
    \begin{equation*}
        \dimh X =\sup\{\dimh \mu\colon \mu\text{ is ergodic}\}.
    \end{equation*}
\end{theorem}
One of the main reasons we restrict our interest to planar attractors is that in higher dimensions, this variational principle does not hold in general even in the self-affine setting \cite{Das2017}. Our result is a generalisation of \cite[Theorem 2]{Gatzouras1997}, where the result was proved for a class of repellers of non-conformal dynamical systems where, when translated to the language of IFSs, {\it domination}, a geometric condition which forces the strongly contracting directions of all of the maps to be uniformly aligned, is imposed on the IFS in addition to the coordinate OSC.

The main idea in the proof is to consider subsystems generated by deep iterates of the original system. Due to a bounded distortion condition, see \cref{lemma:bounded-distortion}, for the coordinate IFSs, the attractors of these subsystems are morally very close to being self-affine sets, and one can then use tools from the self-affine theory to construct sequences of Bernoulli measures on the attractors with dimensions approaching the dimension of the non-linear carpet. In the setting of \cite{Gatzouras1997}, the domination condition ensures that the subsystems are close to being attractors of what are called \emph{Gatzouras-Lalley-carpets} in modern language, whereas the absence of domination means that in our setting, the subsystems resemble the non-dominated \emph{Bara\'nski-carpets}. By utilising the developments in self-affine fractal geometry that have happened after \cite{Gatzouras1997}, namely the work in \cite{Baranski2007}, the proof of \cref{thm:variational-principle} works quite similarly to the proof of \cite[Theorem 2]{Gatzouras1997}.

Our approach also provides a method which could, in theory, be used to obtain numerical estimates for the Hausdorff dimension of non-linear carpets, although we do not pursue this in this paper. Since the statement of the result requires quite a bit of setup, we postpone it to \cref{thm:baranski-like-dimh}, but in essence, we show that the Hausdorff dimension can be obtained as a limit of a sequence of solutions to some optimisation problems.

By combining some ideas used in the proof of \cref{thm:variational-principle}, with a non-linear analogue of an  argument of Ferguson, Jordan and Shmerkin \cite{Ferguson2010}, we get our second result.
\begin{theorem}\label{thm:hausdorff-equals-ml}
    If $X$ is a non-linear carpet which satisfies the coordinate OSC, then
    \begin{equation*}
        \dimh X=\dimml X =\sup\{\diml X'\colon X'\subset X\}.
    \end{equation*}
\end{theorem}
While the proof is similar in spirit to the proof of \cref{thm:variational-principle}, the exponential distortion in the coordinate directions, which can be made arbitrarily small and therefore does not affect the Hausdorff dimension at the limit, can in principle affect the (modified) lower dimension. Essentially, the problem is that, unlike the Hausdorff dimension, the lower dimension is not stable under Hölder maps (or indeed even under Lipschitz maps). In particular, it is in principle possible that the subsystems look dominated in one coordinate direction symbolically, but that there are \emph{some} cylinders, where the bounded distortion causes the cylinder to look dominated in the other direction for long enough to affect the lower dimension. If the non-linear carpet was dominated, this could not happen because the exponential distortion caused by the domination is stronger than the exponential distortion caused by the bounded distortion, so to treat the general case, we use a lemma which shows that, if the non-linear carpet is non-conformal, the measures used to approximate the Hausdorff dimension can be taken to have distinct Lyapunov exponents, see \cref{prop:distinct-lyapunov}. This method, together with the argument in \cite{Ferguson2010}, can then be used to build dominated subsystems for the non-linear carpets with lower dimension arbitrarily close to the dimension of the carpet.

As a straightforward application of \cref{thm:hausdorff-equals-ml} we answer \cref{ques:dasetal}. We need an additional assumption ensuring that the set has large vertical and horizontal fibres, which is needed to satisfy the hyperplane diffusivity condition. The most natural way to express this condition is to divide the maps in the IFS into columns and rows, indexed by the sets $\Lambda_1$ and $\Lambda_2$, respectively. The \emph{column} indexed by $i\in\Lambda_1$ is the collection
\begin{equation*}
    \{S_{i,j}\colon j\in\Lambda_2,\, (i,j)\in\Lambda\},
\end{equation*}
and the \emph{row} indexed by $j\in\Lambda_2$ is defined analogously.

\begin{theorem}\label{thm:bad}
    If $X$ is a non-linear carpet which satisfies the coordinate OSC and has at least two maps in some column and some row, then
    \begin{equation*}
        \dimh(X\cap \Bad_2)=\dimh X.
    \end{equation*}
\end{theorem}

The additional assumptions in \cref{thm:bad} are slightly subtle, and while they ensure that the non-linear carpet is hyperplane diffuse, our method in fact works in some situations where this is not the case. Indeed, to deduce \cref{thm:bad} straight from \cref{thm:hausdorff-equals-ml,schmidt}, we only need to ensure that the \emph{subsets} we use to witness the modified lower dimension in the proof of \cref{thm:hausdorff-equals-ml} are hyperplane diffuse. To be more precise, to prove the result, it is enough that for $\eps>0$ we can find a \emph{closed hyperplane diffuse} subset $X' \subset X$ such that $\diml X' > \dimh X - \eps$. This means that the result also holds in a more general setting where we only need two maps in either a row or a column depending on which coordinate direction dominates in the dimension formula but we chose to state the result as above for simplicity and leave the reader to consider more general statements. The proof of the hyperplane diffusivity is quite simple, but we choose to defer it to \cref{sec:bad} to more explicitly explain the exact situation where our methods work.

One might wonder if the condition can be weakened to only require that $X$ is not contained in any line. Certainly this condition is necessary. However, this assumption is not sufficient for our method of proof. For example, consider a Bedford--McMullen carpet with at most one map in each column but not all maps in the same row.  Clearly the attractor is not contained in a line, but it is also clearly not hyperplane diffuse and, moreover, any subset of the attractor will also fail to be hyperplane diffuse. This can be seen by taking $x$ to be any point on the attractor and letting $V$ be the horizontal line through $x$.  Then, no matter which $\beta>0$ we choose, $B(x,R) \cap X \subseteq V_{\beta R}$ for $R$ sufficiently small due to the increasing eccentricity of the defining maps upon iteration.

\begin{remark}
    Since our Theorem \ref{thm:bad} only applies in the case $d=2$, one might wonder about the status of Question \ref{ques:dasetal} for other values of $d$.  As we mentioned above there are significant challenges moving to higher dimensions in full generality due to the `dimension gap' phenomenon.  However, by replacing the conformal coordinate IFSs with conformal IFSs on higher dimensional spaces, we obtain a natural model where our proofs go through with very little difference. That is for any $d\geq 2$, we may take $(f_i)_{i\in\Lambda_1}$ and $(g_j)_{j\in\Lambda_2}$ to be conformal IFSs in $d_1$ and $d_2$, respectively, with $d_1+d_2=d$, and our methods work with simple modifications for the attractor of the non-conformal IFS $(S_{i,j}\coloneqq (f_i,g_j))_{(i,j)\in\Lambda}$ on $\R^d$. In some sense this model only has `one degree of non-conformality' whereas in full generality one might ask for $(d-1)$ degrees of non-conformality in dimension $d$.  In any case, this simple modification provides answers for the $d \geq2$ cases of Question \ref{ques:dasetal}. The $d=1$ case is of a different flavour, but one could view Proposition \ref{parabolic} as a solution, although we stress that the true spirit of Question \ref{ques:dasetal} requires $d \geq 2$.
\end{remark}

\section{Symbolic spaces}\label{sec:symbolic_spaces}
A substantial portion of the work will be done in symbolic spaces and in this section, we define the symbolic analogues of the non-linear carpets we study, and show how these spaces may be approximated by symbolic spaces which are analogous to self-affine Bara\'nski carpets. We start by setting up some notation.

For a finite index set (also called an \emph{alphabet}) $\Lambda$ we denote by $\Sigma(\Lambda)\coloneqq \Lambda^{\N}$ the associated \emph{symbolic space}. The reader should keep in mind that $\Lambda$ can be an arbitrary index set, for instance, the subsystems we build later in the paper correspond to symbolic spaces with alphabets given by collections of finite words from an initial alphabet. For any $n\in\N$, we call the set $\Lambda^n$ the collection of \emph{words of length $n$}, and denote by $\Lambda^*=\bigcup_{n=0}^{\infty}\Lambda^n$, the collection of all \emph{finite words}. We will denote the elements of both $\Sigma$ and $\Lambda^*$ by $\mtt{i}$, that is $\mtt{i}=i_1i_2\cdots$ or $\mtt{i}=i_1i_2\cdots i_n$ for some $n\in\N$, depending on context.  The length of a word $\mtt{i}\in\Lambda^*$ is denoted by $|\mtt{i}|$, and it is the unique natural number $n$, such that $\mtt{i}\in\Lambda^n$. For two words $\mtt{i}_1,\mtt{i}_2\in\Sigma(\Lambda)$, the \emph{longest common subword} of $\mtt{i}_1$ and $\mtt{i}_2$ is defined by $\mtt{i}_1\wedge \mtt{i}_2\coloneqq\mtt{i_1}|_{k}=\mtt{i_2}|_{k}$, where $k$ is the largest natural number which satisfies $\mtt{i_1}|_{k}=\mtt{i_2}|_{k}$.
Finally, for any collection of real numbers $(c_{i})_{i\in\Lambda}$ indexed by $\Lambda$, we let
\begin{equation*}
    c_{\mtt{i}}=c_{i_1}c_{i_2}\cdots c_{i_{|\mtt{i}|}},
\end{equation*}
for all $\mtt{i}\in\Lambda^*$, and similarly, if $(h_i)_{i\in\Lambda}$ is a collection of self-maps of some metric space, indexed by $\Lambda$, we let
\begin{equation*}
    h_{\mtt{i}}=h_{i_1}\circ h_{i_2}\circ\cdots\circ h_{i_{|\mtt{i}|}}.
\end{equation*}

Given a symbolic space $\Sigma=\Sigma(\Lambda)$, and $\mb{a}=(a_i)_{i\in\Lambda}\in(0,1)^{\Lambda}$, we may define a metric $\rho[\mb{a}]$ on $\Sigma$ by setting
\begin{equation*}
    \rho[\mb{a}](\mtt{i}_1,\mtt{i}_2)=a_{\mtt{i}_1\wedge \mtt{i}_2},
\end{equation*}
for all $\mtt{i}_1,\mtt{i}_2\in\Sigma$. It is easy to see that this defines a metric on $\Sigma$, and the topology induced by this metric coincides with the one generated by the \emph{cylinder sets}, which are defined for any $\mtt{i}\in\Lambda^*$ by
\begin{equation*}
    [\mtt{i}]=\{\mtt{i}'\in\Sigma\colon \mtt{i}'|_{|\mtt{i}|}=\mtt{i}\}.
\end{equation*}

Let us finally recall the definition of Bernoulli measures on symbolic spaces. For an alphabet $\Lambda$, we denote by
 \begin{equation*}
     \mathcal{P}(\Lambda)=\left\{\mb{p}=(p_{i,j})_{(i,j)\in\Lambda}\colon \sum_{(i,j)\in\Lambda}p_{i,j}=1\right\},
 \end{equation*}
 the collection of probability vectors indexed by $\Lambda$, and by $\mathcal{P}^{\circ}(\Lambda)=\{\mb{p}\in\mathcal{P}(\Lambda)\colon 0<p_i<1\,\forall i\in\Lambda\}$. For a given $\mb{p}\in\mathcal{P}(\Lambda)$, we define a pre-measure on the collection of cylinder sets by setting
 \begin{equation*}
     \nu_{\mb{p}}([\mtt{i}])=p_{\mtt{i}},
 \end{equation*}
 for all $\mtt{i}\in\Lambda^*$, and extend this uniquely to a Borel measure on $\Sigma$ using Carathéodory's extension theorem. The resulting measure, which we continue to denote by $\nu_{\mb{p}}$, is called the \emph{Bernoulli measure} associated with the probability vector $\mb{p}$.

\subsection{Symbolic Bara\'nski carpets}
Let us now describe the symbolic analogue of self-affine carpets. Let $\Lambda_1$ and $\Lambda_2$ be finite index sets and let $\Lambda\subset \Lambda_1\times \Lambda_2$. We identify the space $\Sigma=\Sigma(\Lambda)$ with a subset of $\Sigma(\Lambda_1)\times \Sigma(\Lambda_2)$ in the natural way. Given $\mb{a}\in(0,1)^{\Lambda_1}$ and $\mb{b}\in(0,1)^{\Lambda_2}$, we define a metric on $\Sigma$ by
\begin{equation*}
    d[\mb{a},\mb{b}]((\mtt{i}_1,\mtt{j}_1),(\mtt{i}_2,\mtt{j}_2))\coloneqq \max\{\rho[\mb{a}](\mtt{i}_1,\mtt{i}_2),\rho[\mb{b}](\mtt{j}_1,\mtt{j}_2)\}=\max\{\mb{a}_{\mtt{i}_1\wedge \mtt{i}_2},\mb{b}_{\mtt{j}_1\wedge \mtt{j}_2}\}
\end{equation*}
for all $(\mtt{i}_1,\mtt{j}_1),(\mtt{i}_2,\mtt{j}_2)\in\Sigma$. We emphasise that 
\begin{equation*}
    |(\mtt{i}_1,\mtt{j}_1)\wedge (\mtt{i}_2,\mtt{j_2})|=\min\{|\mtt{i}_1\wedge \mtt{i}_2|,|\mtt{j}_1\wedge \mtt{j}_2|\},
\end{equation*}
but in the definition of the metric, the longest common subwords are taken coordinatewise. Therefore, for any $(\mtt{i},\mtt{j})\in\Sigma$, and $r>0$, if we denote by $k$ and $\ell$ the unique integers, which satisfy
\begin{equation*}
    a_{\mtt{i}|_k}<r\leq a_{\mtt{i}|_{k-1}},\text{ and  }b_{\mtt{j}|_\ell}<r\leq b_{\mtt{i}|_{\ell-1}},
\end{equation*}
then the (open) ball with center $(\mtt{i},\mtt{j})\in\Sigma$ and radius $r$ in the metric $d[\mb{a},\mb{b}]$ is the set
\begin{equation*}
    Q(\mtt{i},\mtt{j},r)\coloneqq\{(\mtt{i}',\mtt{j}')\in\Sigma\colon \mtt{i}'|_k=\mtt{i}_k\text{ and }\mtt{j}'|_{\ell}=\mtt{j}|_{\ell}\}.
\end{equation*}
This is precisely the \emph{approximate square} centered at $(\mtt{i},\mtt{j})$, with radius $r$, in the language commonly used in the theory of self-affine sets, see e.g. \cite[Definition 4.3]{Baranski2007}. Due to this connection, we call these spaces \emph{symbolic Bara\'nski carpets}. These spaces play an important role later in the paper, since we may approximate non-linear carpets from the inside by projections of symbolic Bara\'nski carpets. Using the theory of self-affine sets, the Hausdorff dimensions of symbolic Bara\'nski carpets are relatively simple to describe and, crucially, every symbolic Bara\'nski carpet has a dimension maximising Bernoulli measure whose Hausdorff dimension has a relatively simple formula. Let us recall the construction.

For $\mb{p}\in\mathcal{P}\coloneqq\mathcal{P}(\Lambda)$, $i\in\Lambda_1$ and $j\in\Lambda_2$, we let
 \begin{equation*}
     q_i(\mb{p})=\sum_{j:(i,j)\in\Lambda}p_{i,j}\text{ \  and   \ }r_j(\mb{p})=\sum_{i:(i,j)\in\Lambda}p_{i,j}.
 \end{equation*}
 We denote the \emph{Lyapunov exponents} of $\mb{p}$ by
 \begin{equation*}
     \lambda_1(\mb{p},\mb{a})=\sum_{(i,j)\in\Lambda}p_{i,j}\log a_i,\text{ \  and  \ }\lambda_2(\mb{p},\mb{b})=\sum_{(i,j)\in\Lambda}p_{i,j}\log b_j.
 \end{equation*}
 Let us set
 \begin{align*}
     &g_1(\mb{p},\mb{a},\mb{b})=\frac{\sum_{(i,j)\in\Lambda}p_{i,j}\log q_i(\mb{p})}{\lambda_1(\mb{p},\mb{a})}+\frac{\sum_{(i,j)\in\Lambda}p_{i,j}\log p_{i,j}-\sum_{(i,j)\in\Lambda}p_{i,j}\log q_i(\mb{p})}{\lambda_2(\mb{p},\mb{b})}\\
     &g_2(\mb{p},\mb{a},\mb{b})=\frac{\sum_{(i,j)\in\Lambda}p_{i,j}\log r_j(\mb{p})}{\lambda_2(\mb{p},\mb{b})}+\frac{\sum_{(i,j)\in\Lambda}p_{i,j}\log p_{i,j}-\sum_{(i,j)\in\Lambda}p_{i,j}\log r_j(\mb{p})}{\lambda_1(\mb{p},\mb{a})}.
 \end{align*}
 We decompose $\mathcal{P}=\mathcal{P}(\Lambda)$ into sets
 \begin{align*}
     &\mathcal{P}_V=\left\{\mb{p}\in\mathcal{P}\colon \lambda_1(\mb{p},\mb{a})\geq \lambda_2(\mb{p},\mb{b})\right\}\\
     &\mathcal{P}_H=\left\{\mb{p}\in\mathcal{P}\colon \lambda_1(\mb{p},\mb{a})\leq \lambda_2(\mb{p},\mb{b})\right\},
 \end{align*}
 and define
 \begin{equation*}
    g(\mb{p},\mb{a},\mb{b})=\begin{cases}
    g_1(\mb{p},\mb{a},\mb{b}),&\text{ if }\mb{p}\in \mathcal{P}_V,\\
    g_2(\mb{p},\mb{a},\mb{b}),&\text{ if }\mb{p}\in \mathcal{P}_H\setminus \mathcal{P}_V.
     \end{cases}
 \end{equation*}
 It is not difficult to see that if $\mb{p}\in\mathcal{P}_V\cap \mathcal{P}_H$, then $g_1(\mb{p},\mb{a},\mb{b})=g_2(\mb{p},\mb{a},\mb{b})$, and that the function $\mb{p}\mapsto g(\mb{p},\mb{a},\mb{b})$ is continuous on $\mathcal{P}$ \cite{Baranski2007}. While the definition of $g(\mb{p},\mb{a},\mb{b})$ might seem complicated at first glance, the intuition behind it is simple: the first term in $g_1$ is nothing more than the dimension of the projection of the Bernoulli measure $\nu_{\mb{p}}$ onto the first coordinate, and the second term can be thought of as the ``average column dimension'' of the measure, and similarly, the first term in $g_2$ is the dimension of the projection of the Bernoulli measure $\nu_{\mb{p}}$ onto the second coordinate and the second term is the ``average row dimension''.

 We note that while the setting in \cite{Baranski2007} is Euclidean, the methods used are symbolic, in particular, the next proposition follows from \cite[Corollary 5.2]{Baranski2007} and the aforementioned fact that the approximate squares in Definition 4.3 of \cite{Baranski2007} correspond to symbolic balls in the metric $d[\mb{a},\mb{b}]$.
 \begin{proposition}\label{prop:baranski-measure-dim}
     If $(\Sigma,d[\mb{a},\mb{b}])$ is a symbolic Bara\'nski carpet, and $\mb{p}\in\mathcal{P}(\Lambda)$, then
     \begin{equation*}
         \dimh \nu_{\mb{p}}=g(\mb{p},\mb{a},\mb{b})
     \end{equation*}
 \end{proposition}
The next proposition follows by combining Theorem A and Corollary 5.2 in \cite{Baranski2007}.
\begin{proposition}\label{prop:baranski-dim}
     If $(\Sigma,d[\mb{a},\mb{b}])$ is a symbolic Bara\'nski carpet, 
     \begin{equation*}
         \dimh\Sigma=\max_{\mb{p}\in\mathcal{P}(\Lambda)}g(\mb{p},\mb{a},\mb{b}).
     \end{equation*}
     Moreover, there exists $\mb{q}\in\mathcal{P}^{\circ}(\Lambda)$, such that
     \begin{equation*}
         \dimh\Sigma=\dimh\nu_{\mb{q}}=g(\mb{q},\mb{a},\mb{b}).
     \end{equation*}
 \end{proposition}
 While the previous two propositions are enough to prove \cref{thm:variational-principle}, as we will see, the modified lower dimension is more sensitive, and for the proof of \cref{thm:hausdorff-equals-ml}, we require the following lemma.
 \begin{proposition}\label{prop:distinct-lyapunov}
     Let $(\Sigma,d[\mb{a},\mb{b}])$ be a symbolic Bara\'nski carpet and assume that $a_i\ne b_j$ for some $(i,j)\in\Lambda$. Then for any $t<\dimh \Sigma$, there exists $\mb{p}\in\mathcal{P}^{\circ}(\Lambda)$, such that
     \begin{equation*}
         \lambda_1(\mb{p},\mb{a})\ne \lambda_2(\mb{p},\mb{a}),
     \end{equation*}
     and
     \begin{equation*}
         \dimh \nu_{\mb{p}}>t.
     \end{equation*}
 \end{proposition}
\begin{proof}
    Let $\mb{q}\in\mathcal{P}(\Lambda)$ be the probability vector given by \cref{prop:baranski-dim}, which satisfies 
    \begin{equation*}
        \dimh \Sigma=g(\mb{q},\mb{a},\mb{b}).
    \end{equation*}
    If $\lambda_1(\mb{q},\mb{a})\ne\lambda_2(\mb{q},\mb{b})$, there is nothing to prove, so we may assume that $\lambda_1(\mb{q},\mb{a})=\lambda_2(\mb{q},\mb{b})$. Assume without loss of generality that
    \begin{equation*}
        a_{i_0}>b_{j_0},
    \end{equation*}
    for some $(i_0,j_0)\in\Lambda$. Let $(i_1,j_1)\in\Lambda$ be any index which satisfies $a_{i_1}\leq b_{j_1}$; such an index must exist since otherwise we would have
    \begin{equation*}
        \lambda_1(\mb{q},\mb{a})=\sum_{(i,j)\in\Lambda}q_{i,j}\log a_{i}>\sum_{(i,j)\in\Lambda^n}q_{i,j}\log b_{j}=\lambda_2(\mb{q},\mb{b})
    \end{equation*}
    Since $\mb{q}\in\mathcal{P}^{\circ}(\Lambda)$ and $\mb{p}\mapsto g(\mb{p},\mb{a},\mb{b})$ is continuous, we may choose $\delta>0$ small enough such that $0<q_{i_0,j_0}+\delta,q_{i_1,j_1}-\delta<1$, and for the probability vector $\mb{p}=(p_{i,j})_{(i,j)\in\Lambda}$ defined by $p_{i,j}=q_{i,j}$, for all $(i,j)\not\in\{(i_0,j_0),(i_1,j_1)\}$, $p_{i_0,j_0}=q_{i_0,j_0}+\delta$, and $p_{i,j_1}=q_{i_1,j_1}-\delta$, we have
     \begin{equation*}
         g(\mb{p},\mb{a},\mb{b})>t.
     \end{equation*}
     Furthermore, clearly $\mb{p}\in\mathcal{P}^{\circ}(\Lambda)$, and
     \begin{align*}
         \lambda_1(\mb{p},\mb{a})&=\sum_{(i,j)\in\Lambda}q_{i,j}\log a_{i}+\delta\log a_{i_0}-\delta\log a_{i_1}\\
         &>\sum_{(i,j)\in\Lambda}q_{i,j}\log b_{j}+\delta\log b_{j_0}-\delta\log b_{j_1}=\lambda_2(\mb{p},\mb{b}),
     \end{align*}
     as required.
\end{proof}

\subsection{Symbolic self-conformal sets}
Next we recall some basic results on self-conformal iterated function systems, and describe how the symbolic situation changes. Let $(f_i)_{i\in\Lambda}$ be a self-conformal IFS on $[0,1]$, and let $X$ denote its attractor. For a function $f\colon [0,1]\to \R$ we write
\begin{equation*}
    \|f\|=\sup_{x\in[0,1]}|f(x)|.
\end{equation*}
The following basic lemma will be frequently employed throughout the rest of the article.
\begin{lemma}\label{lemma:bounded-distortion}
    There is a constant $0<c\leq1$, such that the following hold:
    \begin{enumerate}
        \item For all $\mtt{i}\in\Lambda^*$ and $x_1,x_2\in X$, we have
        \begin{equation*}
            |f_{\mtt{i}}'(x_1)|\geq c|f_{\mtt{i}}'(x_2)|.
        \end{equation*}
        \item For all $\mtt{i},\mtt{j}\in\Lambda^*$, we have
        \begin{equation*}
            c\|f_{\mtt{i}}'\|\|f_{\mtt{j}}'\|\leq \|f_{\mtt{i}\mtt{j}}'\|\leq \|f_{\mtt{i}}'\|\|f_{\mtt{j}}'\|.
        \end{equation*}
        \item For all $\mtt{i}\in\Lambda^*$, we have
        \begin{equation*}
            c\|f_{\mtt{i}}'\|\leq \diam(f_{\mtt{i}}(X))\leq \|f_{\mtt{i}}'\|.
        \end{equation*}
    \end{enumerate}
\end{lemma}
The first item in the lemma is called the \emph{bounded distortion principle}, and the proof of this and the third item may be found in \cite{Mauldin1996}. The second item follows easily from the first one by using submultiplicativity of the norm and the chain rule. Going forward, we refer to $0<c\leq 1$ in the lemma as the \emph{bounded distortion constant}.

There is a natural metric $\rho$ on the symbolic space $\Sigma=\Sigma(\Lambda)$ which is connected to the geometry of the self-conformal set. This metric is defined for $\mtt{i}_1,\mtt{i}_2\in\Sigma$ by
\begin{equation}\label{eq:rho-def}
    \rho(\mtt{i}_1,\mtt{i}_2)=\|f_{\mtt{i}_1\wedge\mtt{i}_2}'\|.
\end{equation}
It is easy to see using the bounded distortion lemma that the natural projection $\pi\colon \Sigma\to X$ is Lipschitz when $\Sigma$ is equipped with the metric $\rho$, and even bi-Lipschitz, if the IFS satisfies the strong separation condition. However, the symbolic space $(\Sigma,\rho)$ does not fall into the framework of the previous section, since in general we only have
\begin{equation*}
    c^{|\mtt{i}\wedge\mtt{j}|}\|f_{i_1}'\|\|f_{i_2}'\|\cdots\|f_{i_{|\mtt{i}\wedge\mtt{j}|}}'\|\leq\|f_{\mtt{i}\wedge\mtt{i}'}'\|\leq \|f_{i_1}'\|\|f_{i_2}'\|\cdots\|f_{i_{|\mtt{i}\wedge\mtt{j}|}}'\|,
\end{equation*}
by the bounded distortion lemma, and the exponential distortion in the lower bound causes difficulties when bounding the dimension from below. However, we may define a family of metrics on $\Sigma$, which do fall into the framework of the previous section, and approximate the metric $\rho$ arbitrarily well in a Hölder sense. For any $n\in\N$ and $\mtt{i}\in\Lambda^n$, we let
\begin{equation*}
    a_{\mtt{i}}=\|f_{\mtt{i}}'\|.
\end{equation*}
 We let $\Sigma_n=\Sigma(\Lambda^n)$ denote the symbolic space associated to the alphabet $\Lambda^n$, and let $\mb{a}_n=(a_{\mtt{i}})_{\mtt{i}\in\Lambda^n}$. Define a metric $\rho_n$ on $\Sigma_n$ by 
\begin{equation*}
    \rho_{n}=\rho[\mb{a}_{n}].
\end{equation*}
Let $\gamma_n\colon \Sigma_n\to\Sigma$ denote the \emph{concatenation map}, which is the natural bijection between $\Sigma_n$ and $\Sigma$, defined for $\omtt{i}=(\mtt{i}_1)(\mtt{i}_2)\cdots$ by
\begin{equation*}
    \gamma_n(\omtt{i})= \mtt{i_1}\mtt{i_2}\cdots.
\end{equation*}
\begin{lemma}\label{lemma:holder-concatenation}
    There is a constant $C>0$, such that for every $n\in\N$, and $\omtt{i}_1,\omtt{i}_2\in\Sigma_n$, we have
    \begin{equation*}
        c_n\rho_n(\omtt{i}_1,\omtt{i}_2)^{1+\frac{C}{n}}\leq \rho(\gamma_n(\omtt{i}_1),\gamma_n(\omtt{i}_2))\leq \rho_n(\omtt{i}_1,\omtt{i}_2),
    \end{equation*}
    for some constant $c_n>0$.
\end{lemma}
\begin{proof}
    Let $n\in\N$ and notice that for any $\omtt{i}_1,\omtt{i}_2\in\Sigma_n$, we may write
    \begin{equation}\label{eq:wedge}
        \gamma_n(\mtt{i}_1)\wedge\gamma_n(\mtt{i}_2)=\mtt{i}_1\mtt{i}_2\cdots\mtt{i}_k\mtt{i}',
    \end{equation}
    where $k=|\omtt{i}_1\wedge\omtt{i}_2|$, $\mtt{i}_{\ell}\in\Lambda^n$ for all $1\leq \ell\leq k$ and $0\leq|\mtt{i}'|<n$.  We emphasise here that the former largest common subword is taken in terms of the alphabet $\Lambda$, and the latter in terms of $\Lambda^n$. By the submultiplicativity of the supremum norm, we have
    \begin{align*}        \rho(\gamma_n(\mtt{i}_1),\gamma_n(\mtt{i}_2))=\|f_{\gamma_n(\mtt{i}_1)\wedge\gamma_n(\mtt{i}_2)}'\|&\leq \|f_{\mtt{i}_1}'\|\|f_{\mtt{i}_2}'\|\cdots \|f_{\mtt{i}_k}'\|\|f_{\mtt{i}'}'\|\\
        &\leq \|f_{\mtt{i}_1}'\|\|f_{\mtt{i}_2}'\|\cdots \|f_{\mtt{i}_k}'\|=\rho_n(\omtt{i}_1,\omtt{i}_2).
    \end{align*}
    
    For the other inequality, we recall that since the IFS is uniformly contracting, there are constants $0<\lambda_1<\lambda_2<1$, depending only on the functions $f_i$, such that
    \begin{equation*}
        \lambda_1^n\leq \|f_{\mtt{i}}'\|\leq \lambda_2^n,
    \end{equation*}
    for all $\mtt{i}\in\Lambda^n$ and $n\in\N$. In particular, for $\omtt{i}_1,\omtt{i}_2\in\Sigma_n$, we have by bounded distortion that
    \begin{align*}
        \rho(\gamma_n(\mtt{i}_1),\gamma_n(\mtt{i}_2))&=\|f_{\gamma_n(\mtt{i}_1)\wedge\gamma_n(\mtt{i}_2)}\|\geq \lambda_1^{|\mtt{i}'|}\prod_{j=1}^kc\|f_{\mtt{i}_j}'\|\geq \lambda_1^n\prod_{j=1}^k\|f_{\mtt{i}_j}'\|^{1+\frac{\log c}{\log \|f_{\mtt{i}_j}'\|}}\\
        &\geq \lambda_1^n\prod_{j=1}^k\|f_{\mtt{i}_j}'\|^{1+\frac{1}{n}\frac{\log c}{\log \lambda_2}}=\lambda_1^n\rho_n(\omtt{i}_1,\omtt{i}_2)^{1+\frac{1}{n}\frac{\log c}{\log \lambda_2}},
    \end{align*}
    which gives the claim with $c_n=\lambda_1^n$ and $C=\frac{\log c}{\log \lambda_2}$.
\end{proof}

\subsection{Symbolic non-linear carpets}
Finally, we describe the symbolic spaces corresponding to the non-linear carpets discussed in the introduction. For this section, we fix two self-conformal IFSs $(f_i)_{i\in\Lambda_1}$ and $(g_j)_{j\in\Lambda_2}$ on $[0,1]$, with attractors $X_1$ and $X_2$, respectively, and consider the planar IFS $(S_{i,j}=(f_i,g_j))_{(i,j)\in\Lambda}$, where $\Lambda\subset \Lambda_1\times \Lambda_2$. We denote the attractor of the IFS by $X\subset X_1\times X_2$. Note that if $0<c_1,c_2\leq 1$ are the bounded distortion constants given by applying \cref{lemma:bounded-distortion} to the coordinate IFSs $(f_i)_{i\in\Lambda_1}$ and $(g_j)_{j\in\Lambda_2}$, respectively, by taking $c=\min\{c_1,c_2\}$, the bounded distortion lemma applies to both coordinate IFSs with the same constant.

We define a metric $d$ on $\Sigma=\Sigma(\Lambda)$ by setting
\begin{equation*}
    d((\mtt{i}_1,\mtt{j}_1),(\mtt{i}_2, \mtt{j}_2))=\max\{\rho^1(\mtt{i}_1,\mtt{i}_2),\rho^2(\mtt{i}_2, \mtt{j}_2)\},
\end{equation*}
where $\rho^1$ and $\rho^2$ are the metrics on $\Sigma(\Lambda_1)$ and $\Sigma(\Lambda_2)$, respectively, given by \cref{eq:rho-def}. Again, this metric does not quite give us a symbolic Bara\'nski carpet but we may proceed similarly to the previous section, and define for any $n\in\N$ and $(\mtt{i},\mtt{j})\in\Lambda^n$,
\begin{equation}\label{eq:dist}
    \mb{a}_n=(a_{\mtt{i}})_{\mtt{i}\in\Lambda_1^n}\coloneqq(\|f_{\mtt{i}}'\|)_{\mtt{i}\in\Lambda_1^n}\text{ and  }\mb{b}_n=(b_{\mtt{j}})_{\mtt{j}\in\Lambda_2^n}\coloneqq(\|g_{\mtt{j}}'\|)_{\mtt{j}\in\Lambda_2^n}
\end{equation}
and let $\Sigma_n=\Sigma(\Lambda^n)$ denote the symbolic space associated to the alphabet $\Lambda^n$, equipped with the metric $d_n=d[\mb{a}_n,\mb{b}_n]$. As we did with the bounded distortion constant, we may clearly assume that \cref{lemma:holder-concatenation} holds for the same constants for the metrics $\rho^1$ and $\rho[\mb{a}_n]$, as well as $\rho^2$ and $\rho[\mb{b}_n]$.

The concatenation map $\gamma_n\colon \Sigma_n\to\Sigma$ is defined naturally for $(\omtt{i},\omtt{j})\in\Sigma_n$ by
\begin{equation*}
    \gamma_n((\omtt{i},\omtt{j}))= (\gamma_n(\omtt{i}),\gamma_n(\omtt{j})),
\end{equation*}
where we slightly abuse notation by denoting by $\gamma_n$ both the concatenation maps on $\Sigma(\Lambda_1^n)$ and $\Sigma(\Lambda_2^n)$. The next proposition follows easily from \cref{lemma:holder-concatenation}.

\begin{proposition}\label{prop:sigma-dim-limit}
    The symbolic space $\Sigma$ equipped with the metric $d$ satisfies
    \begin{equation*}
        \dimh\Sigma=\lim_{n\to\infty}\dimh \Sigma_n=\lim_{n\to\infty}\max_{\mb{p}\in\mathcal{P}(\Lambda^n)}g(\mb{p},\mb{a}_n,\mb{b}_n),
    \end{equation*}
    where each $\Sigma_n$ is equiped with the metric $d_n=d[\mb{a}_n,\mb{b}_n]$.
\end{proposition}
\begin{proof}
    It follows from \cref{lemma:holder-concatenation}, that for all $n\in\N$ and $(\omtt{i}_1,\omtt{j}_1),(\omtt{i}_2,\omtt{j}_2)\in\Sigma_n$, we have
    \begin{equation*}
        c_nd_n((\omtt{i}_1,\omtt{j}_1),(\omtt{i}_2,\omtt{j}_2))^{1+\frac{C}{n}}\leq d(\gamma_n((\omtt{i}_1,\omtt{j}_1)),\gamma_n((\omtt{i}_2,\omtt{j}_2))\leq d_n((\omtt{i}_1,\omtt{j}_1),(\omtt{i}_2,\omtt{j}_2)).
    \end{equation*}
    Therefore by standard results, see e.g. \cite[Proposition 3.3]{Falconer2014},
    \begin{equation*}
        \limsup_{n\to\infty}\frac{1}{1+\frac{C}{n}}\dimh \Sigma_n\leq \dimh \Sigma \leq \liminf_{n\to\infty}\dimh \Sigma_n,
    \end{equation*}
    which gives the claim by \cref{prop:baranski-dim}.
\end{proof}
The following theorem is the symbolic analogue of \cref{thm:variational-principle}.
\begin{theorem}\label{thm:symbolic-variational-principle}
    For any $\varepsilon>0$, there exists a measure $\nu$ on $\Sigma$, which is ergodic with respect to $\sigma$, and satisfies
    \begin{equation*}
        \dimh \nu\geq \dimh\Sigma-\varepsilon.
    \end{equation*}
\end{theorem}
\begin{proof}
    For every $n\in\N$, let $\mb{q}_n\in\mathcal{P}(\Lambda^n)$ be the probability vector given by applying \cref{prop:baranski-dim} to the symbolic space $(\Sigma_n,d[\mb{a}_n,\mb{b}_n])$, and let $\nu_n\coloneqq \nu_{\mb{q}_n}$ be the corresponding Bernoulli measure on $\Sigma_n$. Define a measure $\mu_n=\nu_n\circ\gamma_n^{-1}$ on $\Sigma$, and then
    \begin{equation*}
        \overline{\mu}_n=\frac{1}{n}\sum_{k=0}^{n-1}\mu_k\circ\sigma^{-k}.
    \end{equation*}
    It follows from $\sigma^k$ ergodicity of $\mu_k$ that $\overline{\mu}_n$ is ergodic with respect to $\sigma$. Moreover, for every $n\in\N$, the measure $\mu_n$ is clearly absolutely continuous with respect to $\overline{\mu}_n$, and therefore
    \begin{equation*}
        \dimh \overline{\mu}_n\geq \dimh \mu_n.
    \end{equation*}
    Applying \cref{lemma:holder-concatenation} as in the proof of \cref{prop:sigma-dim-limit}, we see that 
    \begin{equation*}
        \dimh \mu_n\geq \frac{1}{1+\frac{C}{n}}\dimh \nu_n=\frac{1}{1+\frac{C}{n}}\dimh \Sigma_n,
    \end{equation*}
    so by \cref{prop:sigma-dim-limit}, for any $\varepsilon>0$, we may take $n\in\N$ large enough that
    \begin{equation*}
         \dimh \mu_n\geq\dimh \Sigma-\varepsilon,
    \end{equation*}
    which finishes the proof.
\end{proof}

\section{Proofs of the main results}
In this section we transfer our symbolic results to the geometric setting and prove our main results, \cref{thm:variational-principle,thm:hausdorff-equals-ml}. For the remainder of the section, we fix a non-linear carpet $X$ with a defining IFS $(S_{i,j}=(f_i,g_j))_{(i,j)\in\Lambda}$, which satisfies the coordinate OSC. For any $n\in\N$, we let $\mb{a}_n$ and $\mb{b}_n$ be defined as in \cref{eq:dist}.
\subsection{Hausdorff dimension of non-linear carpets}
Our proof of \cref{thm:variational-principle} is a variant of a standard geometric lemma, which allows us to calculate the local dimensions along approximate squares at typical points. For $(\mtt{i},\mtt{j})\in\Sigma$ and $n\in\N$, write
\begin{equation*}
    \Delta_n=\{(\mtt{i},\mtt{j})\in\Lambda^*\colon \|f_{\mtt{i}}'\|< 2^{-n}\leq \|f_{\mtt{i}^-}'\|,\,\|g_{\mtt{j}}'\|< 2^{-n}\leq  \|g_{\mtt{j}^-}'\|\}.
\end{equation*}
Moreover, for any $(\mtt{i},\mtt{j})\in\Delta_n$, we write
\begin{equation*}
    Q_n(\mtt{i},\mtt{j})=\{(\mtt{i}',\mtt{j}')\in\Sigma\colon \mtt{i'}|_{|\mtt{i}|}=\mtt{i},\text{ and }\mtt{j'}|_{|\mtt{j}|}=\mtt{j}\},
\end{equation*}
for the symbolic approximate square of level $n$, corresponding to $(\mtt{i},\mtt{j})$. Note that $Q_n(\mtt{i},\mtt{j})$ is a ball of radius $2^{-n}$ in the metric $d$.

To pass to the geometric setting, let
\begin{equation*}
    \mathcal{Q}_n=\{f_{\mtt{i}}(U)\times g_{\mtt{j}}(V)\colon (\mtt{i},\mtt{j})\in\Delta_n\}.
\end{equation*}
 Since the coordinate IFSs satisfy the OSC, the collection $\mathcal{Q}_n$ is disjoint, and moreover, there are constants $0<c_1\leq c_2$, such that for any $Q\in\mathcal{Q}_n$, there is a point $(x,y)\in Q$, such that
\begin{equation*}
    B((x,y),c_12^{-n})\subset Q\subset B((x,y),c_22^{-n}).
\end{equation*}
This shows that the collection $\mathcal{Q}_n$ is a \emph{general filtration} in the sense of \cite{Kaenmaki2016}. We let
\begin{equation}\label{eq:Q_n-limit-set}
    E\coloneqq\bigcap_{n=1}^{\infty}\bigcup_{Q\in\mathcal{Q}_n}Q\subset X,
\end{equation}
and for any $(x,y)\in E$, we denote by $Q_n((x,y))$ the unique element of $\mathcal{Q}_n$ that contains $(x,y)$. The next lemma follows from \cite[Proposition 3.1]{Kaenmaki2016}.
\begin{lemma}\label{lemma:dim-along-approx-squares}
    Let $\mu$ be a non-atomic Borel probability measure supported on $E$. Then for $\mu$ almost every $(x,y)\in X$, we have
    \begin{align*}
        &\ldimloc(\mu,(x,y))=\liminf_{n\to\infty}\frac{\log\mu(Q_n((x,y)))}{-n\log 2},\\
        &\udimloc(\mu,(x,y))=\limsup_{n\to\infty}\frac{\log\mu(Q_n((x,y)))}{-n\log 2}.
    \end{align*}
\end{lemma}
We are now ready to prove \cref{thm:variational-principle}.
\begin{theorem}\label{thm:baranski-like-dimh}
    If $X$ is a non-linear carpet which satisfies the coordinate OSC, then for every $\varepsilon>0$, there exists an ergodic measure $\mu$ on $X$, such that
    \begin{equation*}
        \dimh \mu\geq \dimh X-\varepsilon.
    \end{equation*}
    Moreover,
    \begin{equation*}
        \dimh X=\lim_{n\to\infty}\max_{\mb{p}\in\mathcal{P}(\Lambda^n)}g(\mb{p},\mb{a}_n,\mb{b}_n),
    \end{equation*}
    with $\mb{a}_n, \mb{b}_n$ defined by \eqref{eq:dist}. 
\end{theorem}
\begin{proof}
    We may assume without loss of generality that the defining IFS has at least two non-empty columns and two non-empty rows; otherwise it is easy to see that the attractor is contained in a line and is therefore a self-conformal set, in which case the dimension formula simplifies to the classical dimension formula for self-conformal measures.
    
    Let $\varepsilon>0$, let $\nu$ be the measure on $\Sigma$ given by \cref{thm:symbolic-variational-principle}, and write $\mu=\pi_*\nu$. We claim that in order to show that $\dimh\mu\geq \dimh X-\varepsilon$, it suffices to show that for any $(\mtt{i},\mtt{j})\in\Delta_n$,
    \begin{equation}\label{eq:symbolic-measure-eq}
         \mu(f_{\mtt{i}}(U)\times g_{\mtt{j}}(V))= \nu(Q_n(\mtt{i},\mtt{j})).
    \end{equation} 
    Indeed, by \cref{eq:Q_n-limit-set} this immediately shows that $\mu(E)=1$, and it follows from \cref{lemma:dim-along-approx-squares,thm:symbolic-variational-principle} by using the fact that for any $(\mtt{i},\mtt{j})\in\Delta_n$, the set $Q_n(\mtt{i},\mtt{j})$ is a symbolic ball of radius $2^{-n}$,
    that
    \begin{equation*}
        \dimh\mu\geq \dimh\nu\geq \dimh \Sigma-\varepsilon\geq \dimh X-\varepsilon,
    \end{equation*}
    where the last inequality is a consequence of the fact that $\pi\colon \Sigma\to X$ is Lipschitz.
    
    It follows from the OSC for the coordinate IFSs, that $X\cap(f_{\mtt{i}}(\overline{U})\times g_{\mtt{j}}(\overline{V}))= \pi(Q_n(\mtt{i},\mtt{j}))$. This shows that
    \begin{equation*}
        \mu(f_{\mtt{i}}(\overline{U})\times g_{\mtt{j}}(\overline{V}))= \nu(Q_n(\mtt{i},\mtt{j})),
    \end{equation*}
    for all $n\in\N$ and $(\mtt{i},\mtt{j})\in\Delta_n$. Note that the sets $f_{\mtt{i}}(\overline{U})\setminus f_{\mtt{i}}(U)=f_{\mtt{i}}(\overline{U}\setminus U)$ and $g_{\mtt{j}}(\overline{V})\setminus g_{\mtt{j}}(V)=g_{\mtt{j}}(\overline{V}\setminus V)$ are countable, and thus the set $f_{\mtt{i}}(\overline{U})\times g_{\mtt{j}}(\overline{V})\setminus f_{\mtt{i}}(U)\times g_{\mtt{j}}(V)$ is contained in a countable collection of vertical and horizontal lines in $\R^2$. Therefore, in order to prove \cref{eq:symbolic-measure-eq}, it suffices to show that $\mu(V)=0$, for any vertical or horizontal line $V\subset \R^2$. This, however, is immediately clear, since $\mu$ is the average of $k$-step Bernoulli measures for $k=1,\ldots,n$, so its projections to the coordinate axes are averages of ($k$-step) self-conformal measures, which are clearly non-atomic if and only if the non-linear IFS has at least two non-empty rows and columns. The second claim of the proposition follows from \cref{prop:sigma-dim-limit}.
\end{proof}

\subsection{Modified lower dimension of non-linear carpets}
Next we adapt the methods in the previous sections to prove \cref{thm:hausdorff-equals-ml}. For any $n\in\N$, we denote by
\begin{equation*}
    t_n=\max_{\mb{p}\in\mathcal{P}(\Lambda^n)}g(\mb{p},\mb{a}_n,\mb{b}_n),
\end{equation*}
where $\mb{a}_n$ and $\mb{b}_n$ are given by \cref{eq:dist}. We start with two simple lemmas.
\begin{lemma}\label{lemma:sup-gap}
    Assume that there exist $x,y\in X$ and $(i,j)\in\Lambda$, such that
    \begin{equation*}
        |f_i'(x)|\ne |g_j'(y)|.
    \end{equation*}
    Then for infinitely many $n\in\N$, there exists $(\mtt{i},\mtt{j})\in\Lambda^n$, such that
    \begin{equation*}
        \|f_{\mtt{i}}'\|\ne\|g_{\mtt{j}}'\|.
    \end{equation*}
\end{lemma}
\begin{proof}
    Assume without loss of generality that
    \begin{equation*}
        |f_i'(x)|< |g_j'(y)|.
    \end{equation*}
    Since $f_i'$ and $g_j'$ are continuous, for all large enough $n\in\N$, there is $(\mtt{i},\mtt{j})\in\Lambda^n$, such that for all $(x,y)\in[0,1]^2$, we have
    \begin{equation*}
        |f_i'(f_{\mtt{i}}(x))|< |g_j'(g_{\mtt{j}}(y))|.
    \end{equation*}
    We may assume that $\|f_{\mtt{i}}'\|=\|g_{\mtt{j}}'\|$; otherwise we are done. Then for any $x'\in[0,1]$ there exists $y'\in[0,1]$, such that $|f'_{\mtt{i}}(x')|\leq |g'_{\mtt{j}}(y')|$. Let $x'\in[0,1]$ be a point which satisfies $|f_{i\mtt{i}}'(x')|=\|f_{i\mtt{i}}'\|$, which exists by compactness. Then
    \begin{equation*}
        \|f_{i\mtt{i}}'\|=|f_{i\mtt{i}}'(x')|=|f_i'(f_{\mtt{i}}(x'))\|f_{\mtt{i}}'(x'))|< |g_j'(g_{\mtt{j}}(y'))\|g'_{\mtt{j}}(y')|=|g_{j\mtt{j}}'(y')|\leq \|g_{j\mtt{j}}'\|,
    \end{equation*}
    which gives the claim.
\end{proof}
The next lemma follows immediately from the previous one and \cref{prop:distinct-lyapunov}.
\begin{lemma}\label{lemma:lyapunov-exponent-gap-non-linear}
    For any $\varepsilon>0$ and for infinitely many $n\in\N$, there exists $\mb{p}\in\mathcal{P}(\Lambda^n)$, such that
    \begin{equation*}
        \lambda_1(\mb{p},\mb{a}_n)\ne\lambda_2(\mb{p},\mb{b}_n),
    \end{equation*}
    and
    \begin{equation*}
        g(\mb{p},\mb{a}_n,\mb{b}_n)\geq t_n-\varepsilon.
    \end{equation*}
\end{lemma}

We are now ready to prove \cref{thm:hausdorff-equals-ml}.
\begin{proof}[Proof of \cref{thm:hausdorff-equals-ml}]
Let $t<\dimh X$ and $\varepsilon>0$ be small enough that $t(1+\varepsilon)<\dimh X$. Recall that  $\lim_{n\to\infty}t_n=\dimh X$ by \cref{thm:baranski-like-dimh}, so we may choose $n\in\N$ large enough, such that $t(1+\varepsilon)<t_n<\dimh X$. By taking $n\in\N$ larger if needed, and applying \cref{lemma:lyapunov-exponent-gap-non-linear}, we find $\mb{p}\in\mathcal{P}(\Lambda^n)$, such that
\begin{equation*}
    g(\mb{p},\mb{a}_n,\mb{b}_n)> t(1+\varepsilon),
\end{equation*}
and we may assume without loss of generality that $\lambda_1(\mb{p},\mb{a}_n)> \lambda_2(\mb{p},\mb{b}_n)$.

Now for any $k\in\N$, consider the collection
\begin{equation*}
    \Gamma_{n,k}=\{(\omtt{i},\omtt{j})\in\Lambda^{nk}\colon n_{\mtt{i},\mtt{j}}(\omtt{i},\omtt{j})=\lceil kp_{\mtt{i},\mtt{j}}\rceil,\,\forall (\mtt{i},\mtt{j})\in\Lambda^n\},
\end{equation*}
where $n_{\mtt{i},\mtt{j}}(\omtt{i},\omtt{j})\coloneqq\#\{0\leq\ell\leq k-1\colon\sigma^{n\ell}(\omtt{i},\omtt{j})|_{n}=(\mtt{i},\mtt{j})\}$ denotes the number of times the symbol $(\mtt{i},\mtt{j})$ appears in the word $(\omtt{i},\omtt{j})$. Since for some $\delta_n>0$,  $\lambda_1(\mb{p},\mb{a}_n)\geq \lambda_2(\mb{p},\mb{b}_n)+\delta_n$, we have
\begin{align}\label{eq:distinct-lyapunov-iter}
    \sum_{(\mtt{i},\mtt{j})\in\Lambda^n}\lceil kp_{\mtt{i},\mtt{j}}\rceil\log a_{\mtt{i}}&\geq \sum_{(\mtt{i},\mtt{j})\in\Lambda^n}kp_{\mtt{i},\mtt{j}}\log a_{\mtt{i}}+\sum_{(\mtt{i},\mtt{j})\in\Lambda^n}\log a_{\mtt{i}}\\
    &\geq \sum_{(\mtt{i},\mtt{j})\in\Lambda^n}kp_{\mtt{i},\mtt{j}}\log b_{\mtt{j}}+k\delta_n\nonumber+\sum_{(\mtt{i},\mtt{j})\in\Lambda^n}\log a_{\mtt{i}}\\
    &\geq \sum_{(\mtt{i},\mtt{j})\in\Lambda^n}\lceil kp_{\mtt{i},\mtt{j}}\rceil\log b_{\mtt{j}}+k\delta_n\nonumber+\sum_{(\mtt{i},\mtt{j})\in\Lambda^n}\log a_{\mtt{i}}\nonumber\\
    &>\sum_{(\mtt{i},\mtt{j})\in\Lambda^n}\lceil kp_{\mtt{i},\mtt{j}}\rceil\log b_{\mtt{j}}-\log c,\nonumber
\end{align}
 whenever $k>\delta_n^{-1}\left(-\log c-\sum_{(\mtt{i},\mtt{j})\in\Lambda^n}\log a_{\mtt{i}}\right)$, where $0<c\leq1$ is the uniform bounded distortion constant for the coordinate IFSs. Moreover, since the number of times each $(\mtt{i},\mtt{j})\in\Lambda^n$ appears in words of $\Gamma_{n,k}$ is constant, we have that
\begin{equation*}
    a_{\mtt{i}_1\mtt{i}_2\cdots \mtt{i}_k}=a_{\mtt{i}_1'\mtt{i}_2'\cdots \mtt{i}_k'}\eqqcolon a_{n,k}\text{ and }b_{\mtt{j}_1\mtt{j}_2\cdots \mtt{j}_k}=b_{\mtt{j}_1'\mtt{j}_2'\cdots \mtt{j}_k'}\eqqcolon b_{n,k},
\end{equation*}
for all $(\mtt{i}_1\mtt{i}_2\cdots \mtt{i}_k,\mtt{j}_1\mtt{j}_2\cdots \mtt{j}_k),(\mtt{i}_1'\mtt{i}_2'\cdots \mtt{i}_k',\mtt{j}_1'\mtt{j}_2'\cdots \mtt{j}_k')\in\Gamma_{n,k}$, and it follows from \cref{eq:distinct-lyapunov-iter}, that
\begin{equation}\label{eq:domination-gap}
    b_{n,k}<ca_{n,k}.
\end{equation}
Note that by initially taking $n$ large enough, we may also assume that
\begin{equation*}
    (\max_{i\in \Lambda_1}a_i)^{\varepsilon n}\leq c\text{ and }(\max_{j\in \Lambda_2}b_j)^{\varepsilon n}\leq c,
\end{equation*}
and thus for every large enough $k\in\N$, each $(\mtt{i}_1\mtt{i}_2\cdots\mtt{i}_k,\mtt{j}_1\mtt{j}_2\cdots\mtt{j}_k)\in\Gamma_{n,k}$, and every $x\in X$, we have by the chain rule and bounded distortion, that
\begin{equation}\label{eq:S-lb}
    |f_{\mtt{i}_1\mtt{i}_2\cdots\mtt{i}_k}'(x)|=\prod_{j=1}^k|f_{\mtt{i}_j}'(f_{\mtt{i}_{j+1}\cdots \mtt{i}_k}(x))|\geq \prod_{j=1}^kca_{\mtt{i}_j}\geq (a_{\mtt{i}_1\mtt{i}_2\cdots \mtt{i}_k})^{1+\varepsilon},
\end{equation}
where we interpret $f_{\mtt{i}_{j+1}\cdots \mtt{i}_k}(x)=x$, and similarly
\begin{equation}\label{eq:T-lb}
    |g_{\mtt{j}_1\mtt{j}_2\cdots\mtt{j}_k}'(x)|\geq (b_{\mtt{j}_1\mtt{j}_2\cdots \mtt{j}_k})^{1+\varepsilon}.
\end{equation}

Let us now denote by $\Tilde{\Gamma}_{n,k}$ the projections of the strings in $\Gamma_{n,k}$ onto their first coordinates, and let us set
\begin{equation}\label{eq:s_nk}
    s_{n,k}=\frac{\log \#\Tilde{\Gamma}_{n,k}}{-\log a_{n,k}}+\frac{\log \#\Gamma_{n,k}-\log \#\Tilde{\Gamma}_{n,k}}{-\log b_{n,k}}
\end{equation}
Note that  $\#\Tilde{\Gamma}_{n,k}$ corresponds to the number of non-empty columns in the subsystem defined by the alphabet $\Gamma_{n,k}$, and since, by construction, each non-empty column has an equal number of maps, the number of maps in each non-empty column is $\#\Gamma_{n,k}/\#\Tilde{\Gamma}_{n,k}$. Using Stirling approximation precisely as in \cite[Lemma 4.3]{Ferguson2010} shows that $s_{n,k}\to g(\mb{p},\mb{a}_n,\mb{b}_n)$ as $k\to\infty$, and therefore, by choosing $k$ large enough, we have
\begin{equation*}
    \frac{s_{n,k}}{1+\varepsilon}>t.
\end{equation*}

Let us now denote by $\Sigma_{n,k}=\Sigma(\Gamma_{n,k})$ the symbolic space associated with the alphabet $\Gamma_{n,k}$ and let $X_{n,k}=\pi(\Sigma_{n,k})$. To clarify the notation slightly, going forward, the notation without overlines $(\mtt{i},\mtt{j})$ will refer to the letters in the alphabet $\Gamma_{n,k}$, and we use overlines $(\omtt{i},\omtt{j})$ to refer to the words in $\Gamma_{n,k}^*$ and $\Sigma_{n,k}$. The following lemma is the main geometric argument in the proof.
\begin{lemma}
    For any $T\in\Tan(X_{n,k})$, we have
    \begin{equation*}
        \dimh T\geq t.
    \end{equation*}
\end{lemma}
\begin{proof}
    Let $T\in\Tan(X_{n,k})$ and let $x_m\in X_{n,k}$, $r_m>0$ and
    \begin{equation*}
        \frac{X_{n,k}\cap B(x_m,r_m)-x_m}{r_m}\to T,
    \end{equation*}
    in the Hausdorff distance. Denote by $M_{x_m,r_m}\colon \R^2\to\R^2$ the map
    \begin{equation*}
        M_{x_m,r_m}(y)=\frac{y-x_m}{r_m}.
    \end{equation*}
     For each $m\in\N$, let $(\omtt{i}_m,\omtt{j}_m)\in\Sigma_{n,k}$ be such that $\pi(\omtt{i}_m,\omtt{j}_m)=x_m$, and choose $k_{m},\ell_{m}\in\N$, to be the first natural numbers which satisfy
    \begin{equation}\label{eq:r_m-bound}
        \|f_{\omtt{i}_m|_{k_m}}'\|\leq r_m,\text{ and }\|g_{\omtt{j}_m|_{\ell_m}}'\|\leq r_m.
    \end{equation}
    Let us denote by
    \begin{equation*}
        Q_m=\{(\omtt{i},\omtt{j})\in \Sigma(\Gamma_{n,k})\colon \omtt{i}|_{k_m}=\omtt{i}_m|_{k_m}\text{ and }\omtt{j}|_{\ell_m}=\omtt{j}_m|_{\ell_m} \},
    \end{equation*}
    the symbolic approximate square of radius $r_m$ centered at $(\omtt{i},\omtt{j})$, and by
    \begin{equation}\label{eq:geom-approx-square}
        D_m=\pi(Q_m),
    \end{equation}
    the geometric one. It follows from \cref{eq:domination-gap} and \cref{lemma:bounded-distortion}, that for any $(\omtt{i},\omtt{j})\in\Gamma_{n,k}^{\ell}$,
    \begin{equation}\label{eq:k>l}
        \|g_{\omtt{j}}'\|\leq (b_{n,k})^{\ell}< c^{\ell}(a_{n,k})^{\ell}\leq \|f_{\omtt{i}}'\|,
    \end{equation}
    so in particular $k_m\geq \ell_m$ for all $m\in\N$. Also \cref{eq:r_m-bound} implies that $\diam(D_m)\leq r_m$, and since $x_m\in D_m$, we have $D_m\subset B(x_m,r_m)$. In particular, there is $F\subset T$, such that, after passing to a subsequence if needed,
    \begin{equation*}
        F_m\coloneqq\frac{D_m-x_m}{r_m}\to F,
    \end{equation*}
    in the Hausdorff distance.

    We now define a measure $\nu_m$ on $D_m$ in the following way. Start by noting that since $k_m\geq \ell_m$,
    \begin{equation*}
        Q_m=\bigcup_{\omtt{j}'\in \Delta_m(\omtt{i}_m,\omtt{j}_m)}[\omtt{i}_m|_{k_m}]\times [\omtt{j}_m|_{\ell_m}\mtt{j}'],
    \end{equation*}
    where
    \begin{equation*}
        \Delta_m(\omtt{i}_m,\omtt{j}_m)=\{\omtt{j}'\in\Lambda_2^{(k_m-\ell_m)nk}\colon (\omtt{i}_m|_{k_m},\omtt{j}_m|_{\ell_m}\omtt{j}')\in \Gamma_{n,k}^{k_{m}}\}.
    \end{equation*}
    Let us set for each $\omtt{j}'\in \Delta_m(\omtt{i}_m,\omtt{j}_m)$
    \begin{equation*}
        \nu_m(S_{\omtt{i}_m|_{k_m},\omtt{j}_m|_{\ell_m}\mtt{j}'}(X_{n,k}))=\left(\frac{\#\Tilde{\Gamma}_{n,k}}{\#\Gamma_{n,k}}\right)^{k_m-\ell_m}.
    \end{equation*}
    We note that since $\omtt{i}_m|_{k_m}=\omtt{i}_m|_{\ell_m}\omtt{i}'$, for some unique word $\omtt{i}'\in\Tilde{\Gamma}_{n,k}^{k_m-\ell_m}$, and since for each $\mtt{i}\in\Tilde{\Gamma}_{n,k}$, the number of elements $(\mtt{i},\mtt{j})$ that are in $\Gamma_{n,k}$ is $\#\Gamma_{n,k}/\#\Tilde{\Gamma}_{n,k}$, we have
    \begin{equation*}
        \#\Delta_m(\omtt{i}_m,\omtt{j}_m)=\left(\frac{\#\Gamma_{n,k}}{\#\Tilde{\Gamma}_{n,k}}\right)^{k_m-\ell_m},
    \end{equation*}
    so the total mass given to $D_m$ in this first stage is $1$. We then uniformly divide mass among the children of each cylinder in an inductive way, by setting for each $(\mtt{i},\mtt{j})\in\Gamma_{n,k}$,
    \begin{equation*}
        \nu_m(S_{\omtt{i}\mtt{i},\omtt{j}\mtt{j}}(X_{n,k}))=\frac{1}{\#\Gamma_{n,k}}\nu_m(S_{\omtt{i},\omtt{j}}(X_{n,k})),
    \end{equation*}
    whenever $\omtt{i}=\omtt{i}_m|_{k_m}\mtt{i}_1\cdots\mtt{i}_p$ and $\omtt{j}=\omtt{j}_m|_{\ell_m}\mtt{j}'\omtt{j}_1\cdots\mtt{j}_p$, with $\omtt{j}'\in\Delta_m(\omtt{i}_m,\omtt{j}_m)$ and $(\mtt{i}_j,\mtt{j}_j)\in\Gamma_{n,k}$ for all $j=1,\ldots,p$. By setting
    \begin{equation*}
        \mu_m=\nu_m\circ M_{x_m,r_m}^{-1},
    \end{equation*}
    the measure $\mu_m$ is supported on $F_m$. We claim that there are constants $c>0$ and $r_0>0$, independent of $m$, such that
    \begin{equation}\label{eq:unif-frostman}
        \mu_m(B(x,r))\leq cr^t,
    \end{equation}
    for all $x\in B(0,1)$ and $0<r<r_0$. To see that this gives the claim, let $\mu$ be a weak-$\ast$ limit point of the sequence $(\mu_m)_m$, and note that since the support of $\mu_m$ is $F_m$, the support of $\mu$ is contained in $F$, and moreover, for each $x\in Q$ and $r<r_0$, we have by the Portmanteau theorem that
    \begin{equation*}
        \mu(B(x,r))\leq \liminf_{m\to\infty}\mu_m(B(x,r))\leq cr^t.
    \end{equation*}
    Therefore $\dimh T\geq\dimh F\geq \dimh \mu\geq t$.

    It remains to prove \cref{eq:unif-frostman}. For this, let $x\in B(0,1)$, and $0<r<\diam(Q)$, and note that $M_{x_m,r_m}^{-1}(B(x,r))=B(x_m+r_mx,r_mr)$. Denote by $y=x_m+r_mx$ and $r'=r_mr$. Let us set
    \begin{equation*}
        \Delta(r')=\{(\omtt{i},\omtt{j})\in\Gamma_{n,k}^*\colon \|f_{\omtt{i}}'\|< r'\leq \|f_{\omtt{i}^-}'\|,\text{ and }\|g_{\omtt{j}}'\| < r'\leq \|g_{\omtt{j}^-}'\|\},
    \end{equation*}
    and further,
    \begin{align*}
        &\Delta(y,r')=\{(\omtt{i},\omtt{j})\in\Delta(r')\colon \pi(\mathcal{Q}(\omtt{i},\omtt{j}))\cap B(y,r')\ne\emptyset \},
    \end{align*}
    where $Q(\omtt{i},\omtt{j})=\{(\omtt{i}',\omtt{j}')\colon \omtt{i}'|_{|\omtt{i}|}=\omtt{i},\text{ and }\omtt{j}'|_{|\omtt{j}|}=\omtt{j}\}$, denotes the symbolic approximate square associated to $(\omtt{i},\omtt{j})$. Analogously to \cite{Peres2001}, it follows from the fact that the coordinate IFSs satisfy the OSC, that there exists a constant $M_{n,k}<\infty$, independent of $y$ and $r'$, such that $\#\Delta(y,r')\leq M_{n,k}$.
    
    If $(\omtt{i},\omtt{j})\in\Delta(y,r')$, then it follows from \cref{eq:k>l} that $|\mtt{i}|\geq |\mtt{j}|$, and arguing as before, the set $\pi(Q(\omtt{i},\omtt{j}))$ is covered by
    \begin{equation*}
        \left(\frac{\#\Gamma_{n,k}}{\#\Tilde{\Gamma}_{n,k}}\right)^{|\omtt{i}|-|\omtt{j}|},
    \end{equation*}
    cylinders of level $|\omtt{i}|$, and by the definition of the measure $\nu_m$, each of these cylinders has mass
    \begin{equation*}
        \left(\frac{\#\Tilde{\Gamma}_{n,k}}{\#\Gamma_{n,k}}\right)^{k_m-\ell_m}\left(\frac{1}{\#\Gamma_{n,k}}\right)^{|\omtt{i}|-k_m}.
    \end{equation*}
    Therefore,
    \begin{align*}
        \nu_m(\pi(Q(\omtt{i},\omtt{j})))
        &\leq \left(\frac{\#\Gamma_{n,k}}{\#\Tilde{\Gamma}_{n,k}}\right)^{|\omtt{i}|-|\omtt{j}|}\left(\frac{\#\Tilde{\Gamma}_{n,k}}{\#\Gamma_{n,k}}\right)^{k_m-\ell_m}\left(\frac{1}{\#\Gamma_{n,k}}\right)^{|\omtt{i}|-k_m}\\
        &=\left(\frac{1}{\#\Tilde{\Gamma}_{n,k}}\right)^{|\omtt{i}|-k_m}\left(\frac{\#\Tilde{\Gamma}_{n,k}}{\#\Gamma_{n,k}}\right)^{|\omtt{j}|-\ell_m}.
    \end{align*}
    We may write $\omtt{i}=\omtt{i}_m|_{k_m}\omtt{i}'$ and $\omtt{j}=\omtt{j}_m|_{\ell_m}\omtt{j}'$, with $|\omtt{i}'|=|\omtt{i}|-k_m$ and $|\omtt{j}'|=|\omtt{j}|-\ell_m$, and therefore, by the definitions of $\Delta(y,r')$, and using bounded distortion together with the choice of $k_m$ and $\ell_m$, as well as \cref{eq:T-lb,eq:S-lb}, we have
    \begin{equation*}
        r_mr\geq \|f_{\omtt{i}}'\|\gtrsim r_m a_{n,k}^{(1+\varepsilon)(|\omtt{i}|-k_m)},
    \end{equation*}
    and
    \begin{equation*}
        r_mr\geq \|g_{\omtt{j}}'\|\gtrsim r_m b_{n,k}^{(1+\varepsilon)(|\omtt{j}|-\ell_m)},
    \end{equation*}
    with the implicit constants independent of $m$.
    Combining these inequalities with the previous calculation, we have
    \begin{align*}
        \nu_m(\pi(Q(\omtt{i},\omtt{j})))&\lesssim a_{n,k}^{(1+\varepsilon)(|\omtt{i}|-k_m)\frac{\log \#\Tilde{\Gamma}_{n,k}}{-(1+\varepsilon)\log a_{n,k}}}+b_{n,k}^{(1+\varepsilon)(|\omtt{j}|-\ell_m)\frac{\log \#\Gamma_{n,k}-\log \#\Tilde{\Gamma}_{n,k}}{-(1+\varepsilon)\log b_{n,k}}}\\
        &\lesssim r^{\frac{s_{n,k}}{1+\varepsilon}}\leq r^t,
    \end{align*}
    and therefore
    \begin{equation*}
        \mu_m(B(x,r))=\nu_m(B(y,r'))\leq \sum_{(\omtt{i},\omtt{j})\in\Delta(y,r')}\nu_m(\pi(Q(\omtt{i},\omtt{j})))\lesssim M_{n,k}r^t,
    \end{equation*}
    which finishes the proof 
\end{proof}
Recalling that the lower dimension is defined as
\begin{equation*}
    \diml X=\min\{\dimh T\colon T\in\Tan(X)\},
\end{equation*}
the previous lemma shows that $\diml X_{n,k}\geq t$, and therefore $\dimml X\geq t$. Since $t<\dimh X$ was arbitrary, the claim follows.
\end{proof}

\subsection{Badly approximable numbers on non-linear carpets}\label{sec:bad}

The only thing we need to check in order to deduce \cref{thm:bad} straight from the proof of \cref{thm:hausdorff-equals-ml} above is that the sets we used to witness the modified lower dimension in the proof of \cref{thm:hausdorff-equals-ml}  are closed and hyperplane diffuse. The sets we construct are clearly compact, since they are themselves attractors of non-linear carpets, and we argue now that they are also hyperplane diffuse. This is what the additional assumption in the statement of \cref{thm:bad} guarantees.  Indeed, the projection of $X_{n,k}$ onto the first coordinate has a  strictly positive diameter $d_1>0$, which is ensured since $\#\Tilde{\Gamma}_{n,k} >1$ by the assumption that the IFS has a row with at least two maps.  Moreover, all non-empty vertical fibres of $X_{n,k}$ have diameters uniformly bounded below by some $d_2>0$, which is ensured since $\#\Gamma_{n,k}>\#\Tilde{\Gamma}_{n,k}$ and the number of maps used in each column is constant.

Let $R>0$, $x \in X_{n,k}$ and $V$ be an affine hyperplane (a line) in the plane.  Let $D_m$ be the (geometric) approximate square, defined as in \cref{eq:geom-approx-square}, centred at $x$ and of radius $2^{-m}$ where $m$ is chosen uniquely to ensure that 
\[
D_m \subseteq B(x,R) 
\]
and
\[
D_{m-1} \nsubseteq B(x,R).
\]
Then, 
\[
D_m \setminus V_{\beta R} \neq \emptyset,
\]
for $\beta = c\min\{d_1,d_2\}$ for some constant $c>0$ depending only on the original IFS. This completes the proof.

\section{Parabolic Cantor sets}

In this section we observe that the Schmidt game approach can also be used to solve the intersecting with $\Bad_d$ problem for another well-known family of dynamically defined fractals. In \cite{Das2018}, it was observed that the Schmidt game approach works for infinitely generated self-conformal sets, by passing to finitely generated subsystems with large dimension. In this section, we observe that a similar approach works for the non-uniformly contracting parabolic Cantor sets. Even though the main result of this section follows easily from known results, we are unaware of it appearing in the literature and so we briefly discuss the details.

Parabolic Cantor sets were introduced by Urba\'nski \cite{Urbanski1996} and can be thought of as attractors of IFSs acting on the line where maps are allowed to have indifferent (or parabolic) fixed points, where the derivative is 1. We briefly introduce the model. For a differentiable contractive map $f$ on $[0,1]$, we say that $p \in [0,1]$ is a \emph{parabolic point} if $f(p)=p$ and $|f'(p)| = 1$, where we use one sided derivatives at the end points if necessary. Clearly if $h$ is a differentiable contractive map, then it has at most one parabolic point.

 Let $\Lambda$ be a finite index set with at least two elements, and consider an IFS $(f_i)_{i \in \Lambda}$ of differentiable contractive maps acting on $[0,1]$. If at least one of the maps $(f_i)_{i \in \Lambda}$ has a parabolic point, we call the IFS a \emph{parabolic IFS}.  In addition to the above, we make the following assumptions:
\begin{enumerate}
\item For all $i \in \Lambda$,  $f_i$ is $C^{2}$ on $[0,1]$.  \label{it:ifs2}
\item For all $i \in \Lambda$,  $f_i$ has  non-vanishing derivative on  $[0,1]$.  \label{it:bilip2}
\item The IFS $(f_i)_{i\in\Lambda}$ satisfies the OSC with the open set $(0,1)$. \label{it:grid2}
\end{enumerate}
We call attractors of parabolic IFSs satisfying \cref{it:ifs2}, \cref{it:bilip2} and \cref{it:grid2} \emph{parabolic Cantor sets}. In particular, if the attractor is not the whole interval then it is a topological Cantor set. The following observation is the main result of this section.

%Let $\Sigma = \Sigma(\Lambda)$ the symbolic space over the alphabet $\Lambda$ and let $\Lambda^*$ denote the collection of all finite words. If $X$ denotes the attractor of the parabolic IFS, then the natural projection $\pi: \Sigma \to X$ as defined as earlier by
% \[
% \pi(\mtt{i}) = \lim_{n\to\infty} f_{\mtt{i}|_n}(0),
% \]
% is a well defined surjection.  \textcolor{red}{can we get rid of the natural projection, symbolic space etc, now that the attractor is defined from the start?}

\begin{proposition} \label{parabolic}
    Let $X$ be a parabolic Cantor set.  Then 
    \[
    \dimh \Bad_1 \cap X = \dimh X.
    \]
\end{proposition}

\begin{proof}
    The result follows by two simple observations.  First, given any $\eps>0$, there exists a subset $Y \subseteq X$ such that
    \[
    \dimh Y > \dimh X - \eps
    \]
    and such that $Y$ is a self-conformal subset of $X$ generated by a \emph{uniformly contracting} sub IFS of a (high iterate) of the parabolic IFS defining $X$.  This fact is essentially folklore, but was proved explicitly in, for example, \cite[Section 3]{FraserJurga2025}.  Next, since $Y$ is the attractor of a uniformly contracting $C^2$ IFS, it is quasi-self-similar and 
    \[
    \diml Y = \dimh Y.
    \]
    For this, see \cite[Corollary 6.4.4]{Fraser21}.  Then by \cref{schmidt} 
    \[
    \dimh X \geq \dimh \Bad_1 \cap X \geq \dimh \Bad_1 \cap Y \geq \diml Y = \dimh Y > \dimh X-\eps 
    \]
    and letting $\eps \to 0$ completes the proof. 
\end{proof}

In the above, one may dispense with the $C^2$ assumption by only assuming $C^{1+\alpha}$ along with another technical condition, see \cite[Equation (2.3)]{FraserJurga2025}, but we leave the precise formulation of this to the reader.

\bibliographystyle{abbrvurl}
\bibliography{bibliography}

\end{document}